\documentclass[11pt]{amsart}

\usepackage[a4paper,hmargin={2.5cm,2.5cm},vmargin={2.5cm,2.5cm}]{geometry}
\usepackage[matrix,arrow,curve,cmtip]{xy}
\usepackage{latexsym}
\usepackage{graphicx}
\usepackage{amssymb}
\usepackage{amsthm}
\usepackage{txfonts}
\usepackage[mathscr]{euscript}
\usepackage{titlesec}

\titleformat{\section}[hang]%
{\bf\filcenter\large}{\thesection.}{1ex}{}%
\titleformat{\subsection}[runin]%
{\bfseries\normalsize}{}{0ex}{}

\newcommand\adjunction[2]{\xymatrix@=8ex{\ar@{}[r]|{\perp}\ar@<-1mm>@/_2mm/[r]_{{#2}} & \ar@<-1mm>@/_2mm/[l]_{{#1}}}}

\theoremstyle{plain}
\newtheorem{theorem}{Theorem}[section]
\newtheorem{lemma}[theorem]{Lemma}
\newtheorem{proposition}[theorem]{Proposition}
\newtheorem{corollary}[theorem]{Corollary}

\theoremstyle{definition}
\newtheorem{definition}[theorem]{Definition}
\newtheorem{examples}[theorem]{Examples}
\newtheorem{example}[theorem]{Example}

\theoremstyle{remark}
\newtheorem{remark}[theorem]{Remark}

\newenvironment{eqcond}{\begin{enumerate}\renewcommand{\theenumi}{\roman{enumi}}}{\end{enumerate}}
\renewcommand{\theenumi}{\arabic{enumi}}

\newcommand{\field}[1]{\mathbb{#1}}

\newcommand{\R}{\field{R}}

\newcommand{\fra}{\mathfrak{a}}

\newcommand{\frf}{\mathfrak{f}}
\newcommand{\frj}{\mathfrak{j}}
\newcommand{\frp}{\mathfrak{p}}
\newcommand{\frq}{\mathfrak{q}}

\newcommand{\frw}{\mathfrak{w}}
\newcommand{\frv}{\mathfrak{v}}
\newcommand{\frx}{\mathfrak{x}}
\newcommand{\fry}{\mathfrak{y}}

\newcommand{\calA}{\mathcal{A}}
\newcommand{\calB}{\mathcal{B}}

\newcommand{\frX}{\mathfrak{X}}

\DeclareMathOperator{\ssup}{sup}

\DeclareMathOperator{\ev}{ev}
\DeclareMathOperator{\id}{id}
\DeclareMathOperator{\Id}{Id}
\DeclareMathOperator{\ForgetToV}{S}
\DeclareMathOperator{\ForgetToVAd}{A}
\DeclareMathOperator{\MFunctor}{M}
\DeclareMathOperator{\Discrete}{D}

\DeclareMathOperator{\pt}{pt}
\DeclareMathOperator{\yoneda}{y}

\newcommand{\mate}[1]{\,^\ulcorner\! #1^\urcorner}
\newcommand{\umate}[1]{\,_\llcorner\! #1_\lrcorner}
\newcommand{\fspstr}[2]{\llbracket #1,#2\rrbracket}

\newcommand{\catfont}[1]{\mathsf{#1}}

\newcommand{\V}{\catfont{V}}

\newcommand{\two}{\catfont{2}}

\newcommand{\SET}{\catfont{Set}}
\newcommand{\REL}{\catfont{Rel}}
\newcommand{\TOP}{\catfont{Top}}
\newcommand{\FRM}{\catfont{Frm}}

\newcommand{\AP}{\catfont{App}}

\newcommand{\ORD}{\catfont{Ord}}
\newcommand{\MET}{\catfont{Met}}

\newcommand{\Mat}[1]{#1\text{-}\catfont{Mat}}
\newcommand{\Mod}[1]{#1\text{-}\catfont{Dist}}
\newcommand{\Map}{\catfont{Map}}
\newcommand{\UMat}[1]{#1\text{-}\catfont{URel}}
\newcommand{\Frm}[1]{#1\text{-}\catfont{Frm}}
\newcommand{\Cat}[1]{#1\text{-}\catfont{Cat}}
\newcommand{\Cont}[1]{#1\text{-}\catfont{Cont}}

\newcommand{\Gph}[1]{#1\text{-}\catfont{Gph}}

\def\to{\longrightarrow}
\newcommand{\natto}{\Rightarrow}
\newcommand{\relto}{{\longrightarrow\hspace*{-2.8ex}{\mapstochar}\hspace*{2.6ex}}}
\newcommand{\modto}{{\longrightarrow\hspace*{-2.8ex}{\circ}\hspace*{1.2ex}}}
\newcommand{\kto}{\relbar\joinrel\rightharpoonup}
\newcommand{\krelto}{\,{\kto\hspace*{-2.5ex}{\mapstochar}\hspace*{2.6ex}}}
\newcommand{\kmodto}{\,{\kto\hspace*{-2.8ex}{\circ}\hspace*{1.3ex}}}

\newcommand{\kleisli}{\circ}

\newcommand{\homkleislileft}{\multimap} 
\newcommand{\homkleisliright}{\multimapinv} 

\newcommand{\homcompleft}{\multimapdot} 
\newcommand{\homcompright}{\multimapdotinv} 

\newcommand{\Txi}{T_{\!\!_\xi}}
\newcommand{\Uxi}{U_{\!\!_\xi}}

\newcommand{\monadfont}[1]{\mathbb{#1}}

\newcommand{\mT}{\monadfont{T}}

\newcommand{\mU}{\monadfont{U}}

\newcommand{\mL}{\monadfont{L}}

\newcommand{\umonad}{(U,m,e)}

\newcommand{\wmonad}{(L,m,e)}

\newcommand{\theoryfont}[1]{\mathscr{#1}}

\newcommand{\Tth}{\theoryfont{T}}
\newcommand{\Ith}{\theoryfont{I}}
\newcommand{\Uth}{\theoryfont{U}}
\newcommand{\Wth}{\theoryfont{L}}
\newcommand{\toptheory}{(\mT,\V,\xi)}

\newcommand{\doo}[1]{\overset{\centerdot}{#1}}

\newcommand{\op}{\mathrm{op}}
\newcommand{\co}{\mathrm{co}}
\def\:{\colon}
\def\tensor{\otimes}

\usepackage[a4paper,hypertexnames=false]{hyperref}

\begin{document}

\title{Towards Stone duality for topological theories}

\author{Dirk Hofmann}
\thanks{Partial financial assistance by Unidade de Investiga\c{c}\~{a}o e Desenvolvimento Matem\'{a}tica e Aplica\c{c}\~{o}es da Universidade de Aveiro/FCT and the project MONDRIAN (under the contract PTDC/EIA-CCO/108302/2008) is gratefully acknowledged}
\address{Departamento de Matem\'{a}tica\\ Universidade de Aveiro\\3810-193 Aveiro\\ Portugal}
\email{dirk@ua.pt}

\author{Isar Stubbe}
\address{Laboratoire de Math\'ematiques Pures et Appliqu\'ees \\ Universit\'e du Littoral-C\^ote d'Opale \\ France}
\email{isar.stubbe@lmpa.univ-littoral.fr}

\date{Submitted 1 April 2010, revised 27 October 2010.}

\begin{abstract}
In the context of categorical topology, more precisely that of $\Tth$-categories \cite{Hof_TopTh}, we define the notion of $\Tth$-colimit as a particular colimit in a $\V$-category. A complete and cocomplete $\V$-category in which limits distribute over $\Tth$-colimits, is to be thought of as the generalisation of a (co-)frame to this categorical level. We explain some ideas on a $\Tth$-categorical version of ``Stone duality'', and show that Cauchy completeness of a $\Tth$-category is precisely its sobriety. 
\end{abstract}



\maketitle

\section*{Introduction}

Let $X$ be a topological space, then $\Omega(X)$, its collection of open subsets, is a {\em frame}: a complete lattice in which finite infima distribute over arbitrary suprema. If $f\:X\to Y$ is a continuous function between topological spaces, then its inverse image $\Omega(f)\:\Omega(Y)\to\Omega(X)$ is a {\em frame homomorphism}, i.e.\ a (necessarily order-preserving) function that preserves finite infima and arbitrary suprema. Thus we obtain a contravariant functor, $\Omega\:\TOP^\op\to\FRM$, from the category of topological spaces and continuous functors to that of frames and frame homomorphisms. It is well known that this functor admits a left adjoint
$$\TOP^\op\adjunction{\pt}{\Omega}\FRM$$
which assigns to any frame $F$ the topological space $\pt(F)$ of its {\em points}: it is the set $\FRM(F,\two)$ of frame homomorphisms from $F$ to the two-element chain, with open subsets $\{\{p\in\pt(F)\mid p(a)=1\}\mid a\in F\}$. If the natural continuous comparison $\eta_X\:X\to\pt(\Omega(X))$ is bijective (in which case it actually is a homeomorphism), then $X$ is said to be {\em sober}. (And because $X$ is $T_0$ if and only if $\eta_X$ is injective, we get that a $T_0$ space is sober if and only if $\eta_X$ is surjective.) Much more can be said about the interplay between topological spaces and frames; we refer to the classic \cite{Joh_StoneSp}.

Since M. Barr's work \cite{Bar_RelAlg} we know that topological spaces and continuous functions are precisely the lax algebras and their lax homomorphisms for the lax extension to $\REL$ of the ultrafilter monad on $\SET$. (The algebras for the monad itself are the compact Hausdorff spaces.) With this in mind, in recent years others have studied more generally the lax extension of monads $T\:\SET\to\SET$ to the category $\Mat{\V}$ of matrices with elements in a quantale $\V$ \cite{CH_TopFeat,CT_MultiCat,SEAL_LaxAlg}: the lax algebras, often referred to as $(T,\V)$-categories or $(T,\V)$-algebras in those references, but we shall call them simply $\Tth$-categories as in \cite{Hof_TopTh}, are then to be thought of as ``topological categories''. Examples include, beside topological spaces, also approach spaces, $\V$-enriched categories, metric spaces, multicategories, and more.

Altogether this then raises a natural question: how should we define ``$\Tth$-frames'' as the analogue of frames? Is there any hope for a duality between $\Tth$-categories and ``$\Tth$-frames'', generalising that between topological spaces and frames? This is the problem that we address in this paper. 

More exactly, we study the generalisation of the notion of {\em co-frame} in the context of $\Tth$-categories. To give an idea of the main difficulty, reconsider the definition: a co-frame is a complete ordered set in which {\em finite} suprema distribute over (arbitrary) infima. To translate this statement to the context of $\V$-enriched categories, we know that infima and suprema will become enriched limits and colimits, and the distributivity will be expressed by a certain functor being continuous (see e.g.\ \cite{Kelly_Schmitt} for examples in the realm of enriched categories). But how should we translate the {\em finiteness} of the involved suprema/colimits? This is precisely the point where, besides the categorical data (i.e.\ the categories enriched in a quantale $\V$), we must make use of the additional topological data (i.e.\ the monad $T$ on $\SET$): in Definition \ref{T-suprema} we thus propose the notion of ``$\Tth$-supremum'' in a $\V$-category, to be thought of as a ``finite supremum'', where the {\em finiteness} relates (perhaps not surprisingly) to the notion of {\em compactness} relative to the given monad $T$, as developed in \cite{Hof_TopTh}. (More generally, we define ``$\Tth$-colimits'' in Definition \ref{T-colimit}; and a $\V$-category is $\Tth$-cocomplete if and only if it is tensored and has $\Tth$-suprema, cf.\ Proposition \ref{T-cocomplete}.) In Definition \ref{TFrm} we then define a ``$\Tth$-frame'' to be a complete $\V$-category in which $\Tth$-suprema suitably distribute over limits. (Note that we speak of $\Tth$-frames even though we generalise the notion of co-frames.) Of course, these notions are so devised that, when applied to $\V=\two$ (= the two-element chain) and $T=U$ (= the ultrafilter monad), so that $\Tth$-categories are precisely topological spaces, we eventually recover the ordinary co-frames, as shown in Proposition \ref{ultrafilter_case} and further on. For the general case, we show in Corollaries \ref{functor1} and \ref{functor2} that there is a pair of functors
$$\Cat{\Tth}^\op\xymatrix@=8ex{\ar@<-1mm>@/_2mm/[r]_{\Omega} & \ar@<-1mm>@/_2mm/[l]_{\pt}}\Frm{\Tth}.$$
Even though at this point we are unable to prove that these are adjoint, we do show in Proposition \ref{nattrans} that there is a natural transformation $\Id\natto\pt\circ\Omega$; and in Theorem \ref{MainThm} we do prove that the $\Tth$-categories for which the comparison $X\to\pt(\Omega(X))$ is surjective, are precisely those which are {\em Cauchy complete}, which is indeed the expected generalisation of sobriety \cite{Law_Metric,CH_Compl}. 

We see this work as a first step towards an eventual ``Stone duality'' for $\Tth$-categories, and hope that by explaining our ideas, further research on this topic shall be stimulated.

\section{The setting: strict topological theories}

M. Barr \cite{Bar_RelAlg} showed in what sense topological spaces can be thought of as algebras: If we write $U\:\SET\to\SET$ for the ultrafilter monad, with multiplication $m\:U\circ U\natto U$ and unit $e\:\Id_{\SET}\natto U$, then its category of Eilenberg-Moore algebras is precisely that of compact Hausdorff spaces and continuous maps \cite{Man_TriplCompAlg}. But $U$ admits a {\em lax extension} to $\REL$, the quantaloid of sets and relations: define $U'\:\REL\to\REL$ to agree with $U$ on the objects, and for a relation $r\:X\relto Y$ with projection maps $p\:R\to X$ and $q\:R\to Y$ put $U'(r)=Uq\cdot(Up)^\circ$. Then $U'$ is still a functor, and the unit and the multiplication of the ultrafilter monad become oplax natural transformations. Hence $(U',m,e)$ is no longer a monad but rather a {\em lax monad}. Nevertheless, the {\em lax algebras} for $(U',m,e)$ are precisely topological spaces, and the lax algebra homomorphisms turn out to be exactly the continuous maps.

This situation can be generalised, not only by considering other monads $(T,m,e)\:\SET\to\SET$ besides the ultrafilter monad, but also by studying their lax extensions to quantaloids $\Mat{\V}$ of matrices with elements in a commutative quantale $\V$. (In this paper, $\V=(\V,\bigvee,\tensor,k)$ will always stand for a commutative, unital quantale: $(\V,\bigvee)$ is a complete lattice, in which the supremum of a family $(x_i)_{i\in i}$ is written as $\bigvee_ix_i$, together with an associative and commutative operation $\V\times\V\to\V\:(x,y)\mapsto x\tensor y$ with two-sided unit $k\in\V$, such that both $x\tensor-$ and $-\tensor y$ preserve arbitrary suprema. When one takes $\V$ to be the two-element chain, then it turns out that $\Mat{\V}$ is simply $\REL$, as we explain further on.) The lax algebras for a lax extension of $T$ to $\Mat{\V}$ are then to be thought of as ``topological categories''.  Of course one has to put conditions on the involved monad and quantale to prove results (in fact, to even define a lax extension and its lax algebras). Over the last decade, several categorical topologists have considered different conditions on $T$ and $\V$ \cite{CH_ExtLax,SEAL_LaxAlg,SEAL_Kleisli}; in this paper we shall use, up to a slight rephrasing, the notion of {\em strict topological theory} as recently put forward in \cite{Hof_TopTh}.

\begin{definition}
A \emph{strict topological theory} $\Tth=\toptheory$ consists of:
\begin{enumerate}
\item a monad $\mT=(T,m,e)$ on $\SET$ (with multiplication $m$ and unit $e$),
\item a commutative quantale $\V=(\V,\bigvee,\tensor,k)$,
\item a function $\xi\:T(\V)\to\V$,
\end{enumerate}
such that
\begin{enumerate}\renewcommand{\theenumi}{\alph{enumi}}
\item $T$ sends pullbacks to weak pullbacks and each naturality square of $m$ is a weak pullback (in other words, $T$ and $m$ satisfy the B\'enabou-Beck-Chevalley condition),
\item\label{condition} $(\V,\xi)$ is a $\mT$-algebra and the monoid structure on $\V$ in $(\SET,\times,1)$ lifts to monoid structure on $(\V,\xi)$ in $(\SET^\mT,\times,1)$,
\item writing $P_{\V}\:\SET\to\ORD$ for the functor that sends a function $f\:X\to Y$ to the left adjoint of the ``inverse image'' $f^{-1}\:\V^Y\to\V^X\:\varphi\mapsto\varphi\cdot f$ (where $\V^X$ is the set of functions from $X$ to $\V$, with pointwise order), the functions $\xi_X\:V^X\to V^{T(X)}\:f\mapsto \xi\cdot T(f)$ (for $X$ in $\SET$) are the components of a natural transformation $(\xi_X)_X\:P_{\V}\natto P_{\V}\circ T$.
\end{enumerate}
\end{definition}
Regarding condition (\ref{condition}) in the above definition, note that a quantale $\V$ is, in particular, a {\it set} equipped with {\em functions} $\V\times\V\to\V\:(x,y)\mapsto x\tensor y$ and $1\to\V\:*\mapsto k$ (where $1=\{*\}$ is a generic singleton) satisfying (diagrammatic) associativity and unit axioms; put briefly, $(\V,\tensor,k)$ is a monoid in the cartesian category $\SET$. But now we ask for a function $\xi\:T(\V)\to\V$ making $(\V,\xi)$ a {\em $\mT$-algebra}, hence it is natural to require that the functions $(x,y)\mapsto x\tensor y$ and $*\mapsto k$ are in fact {\em $\mT$-homomorphisms}, that is, the following diagrams have to commute:
\[
\xymatrix@C=10ex{
T(\V\times \V)\ar[r]^{T(-\tensor-)}\ar[d]_{\langle\xi\cdot T\pi_1,\xi\cdot T\pi_2\rangle} & T(\V)\ar[d]^{\xi} & T(1)\ar[d]^{!}\ar[l]_{T(k)} \\
\V\times\V\ar[r]_{-\tensor-} & \V & 1\ar[l]^{k}}
\]
Put differently, the monoidal structure $(\V,\tensor,k)$ must lift from $\SET$ to the cartesian category $\SET^{\mT}$ of $\mT$-algebras and homomorphisms. Moreover, it then follows -- as shown in \cite[Lemma 3.2]{Hof_TopTh} -- that the {\it closed} structure on $\V$, in other words, the ``internal hom'' defined by $x\tensor y\leq z\iff x\leq\hom(y,z)$, then automatically satisfies
\[
\xymatrix{T(\V\times\V)\ar[rr]^-{T(\hom)}\ar[d]_{\langle\xi\cdot T\pi_1,\xi\cdot T\pi_2\rangle}\ar@{}[drr]|{\ge} && T\V\ar[d]^\xi\\ \V\times\V\ar[rr]_-{\hom} && \V}
\]

\begin{examples}\label{ExTheories} The leading examples of strict topological theories are:
\begin{enumerate}
\item the trivial theory: For any quantale $\V$ we can consider the theory whose monad-part is the identity monad on $\SET$ and for which the required $\xi\:\V\to\V$ is the identity function. We write this trivial strict topological theory as $\Ith_{\V}$.
\item\label{ExTop} the classical ultrafilter theory: Let $\V$ be the 2-element chain $\two$ (to be thought of as the ``classical truth values''), and consider the ultrafilter monad $\mU=\umonad$ on $\SET$. Together with the obvious function $\xi\:U(\two)\to\two$ this makes up a strict topological theory which we write as $\Uth_{\two}$.
\item\label{ExApp} the metric ultrafilter theory: Let $\V$ be the quantale $([0,\infty],\bigwedge,+,0)$ of extended non-negative real numbers \cite{Law_Metric}, and consider again the ultrafilter monad $\mU=\umonad$ on $\SET$. Together with the function 
\[
\xi\:U([0,\infty])\to[0,\infty],\;\;\frx\mapsto\bigwedge\{v\in[0,\infty]\mid[0,v]\in\frx\}
\]
this makes up a strict topological theory, written $\Uth_{[0,\infty]}$.
\item\label{GenUltraTh} the general ultrafilter theory: If $\V$ is any commutative and integral quantale (meaning that $a\tensor b=b\tensor a$ and that $k=\top$) which is completely distributive, then the ultrafilter monad $\mU=\umonad$ on $\SET$ together with the function 
$$\xi\:U(\V)\to\V\:\frx\mapsto\bigwedge_{A\in\frx}\bigvee A$$
is a strict topological theory provided that $\otimes\:\V\times\V\to\V$ is continuous with respect to the compact Hausdorff topology $\xi$ on $\V$. This generalises the two previous examples; details are in \cite{Hof_TopTh}.
\item\label{WordTh} the word theory: For any quantale $\V$, the word monad $\mL=\wmonad$ on $\SET$ together with the function
$$\xi\:L(\V)\to\V\:(v_1,\ldots,v_n)\mapsto v_1\otimes\ldots\otimes v_n\ ,\ (\ )\mapsto k$$
determine a strict topological theory $\Wth_{\V}$.
\end{enumerate}
\end{examples}

In what follows we shall write $\Mat{\V}$ for the quantaloid of $\V$-matrices: its objects are sets, an arrow $r\:X\relto Y$ is a ``matrix'' whose entries are elements of $\V$, indexed by $Y\times X$. The composition of $r\:X\relto Y$ with $s\:Y\relto Z$ is $s\cdot r\:X\relto Z$ whose $(z,x)$-th element is $\bigvee_{y\in Y}s(z,y)\tensor r(y,x)$; the identity on a set $X$ is the obvious diagonal matrix, with $k$'s on the diagonal and $\bot$'s elsewhere. It is the elementwise supremum of parallel matrices that finally makes $\Mat{\V}$ a quantaloid (in fact, it is the free direct-sum completion of $\V$ in the category of quantaloids). As any quantaloid, $\Mat{\V}$ is biclosed (some authors say ``left- and right-closed'', others say simply ``closed''), in the sense that for any matrix $r\:X\relto Y$ and any object $Z$, both order-preserving functions $-\cdot r\:\Mat{\V}(Y,Z)\to\Mat{\V}(X,Z)$ and $r\cdot-\:\Mat{\V}(Z,X)\to\Mat{\V}(Z,Y)$ admit right adjoints: we shall write 
\[s\cdot r\leq t\iff s\leq t \homcompright r\mbox{ \ \ \ and \ \ \ } r\cdot p\leq q\iff p\leq r\homcompleft q\]
for these liftings and extensions. Finally we mention that mapping a matrix $r\:X\relto Y$ to $r^\circ\:Y\relto X$, defined by $r^\circ(x,y):=r(y,x)$, defines an involution $(-)^\circ\:\Mat{\V}^\op\to\Mat{\V}$.

A topological theory $\Tth=\toptheory$ allows for a lax extension of the functor $T\:\SET\to\SET$ to a 2-functor $\Txi\:\Mat{\V}\to\Mat{\V}$ as follows: we put $\Txi X=TX$ for each set $X$, and 
\begin{equation*}
\Txi r\:TY\times TX\to\V:
(\fry,\frx)\mapsto\bigvee\left\{\xi\cdot Tr(\frw)\;\Bigl\lvert\;\frw\in T(Y\times X), T\pi_1(\frw)=\fry,T\pi_2(\frw)=\frx\right\}
\end{equation*}
for each $\V$-matrix $r\:X\relto Y$. Furthermore, we have $\Txi(r^\circ)=\Txi(r)^\circ$ (and we write $\Txi r^\circ$) for each $\V$-matrix $r\:X\relto Y$, $m$ becomes a natural transformation $m\:\Txi\Txi\natto\Txi$ and $e$ an op-lax natural transformation $e\:\Id\natto\Txi$, i.e.\ $e_Y\cdot r\leq \Txi r\cdot e_X$ for all $r\:X\relto Y$ in $\Mat{\V}$.

A $\V$-matrix of the form $\alpha\:X\relto TY$ we call \emph{$\Tth$-matrix} from $X$ to $Y$, and write $\alpha\:X\krelto Y$. For $\Tth$-matrices $\alpha\:X\krelto Y$ and $\beta\:Y\krelto Z$ we define as usual the \emph{Kleisli composition}
\[\beta\kleisli\alpha:=m_X\cdot\Txi \beta\cdot\alpha.\]
This composition is associative and has the $\Tth$-matrix $e_X\:X\krelto X$ as a lax identity: $a\kleisli e_X\ge a$ and $e_Y\kleisli a= a$ for any $a\:X\krelto Y$. 

We now come to the definition of the ``topological categories'' that we were after in the first place.
\begin{definition}
Let $\Tth$ be a strict topological theory. A \emph{$\Tth$-graph} is a pair $(X,a)$ consisting of a set $X$ and a $\Tth$-matrix $a\:X\krelto X$ satisfying $e_X\le a$. A \emph{$\Tth$-category} $(X,a)$ is a $\Tth$-graph such that moreover $a\circ a\leq a$. Given two $\Tth$-graphs (resp.\ $\Tth$-categories) $(X,a)$ and $(Y,b)$, a function $f\:X\to Y$ is a \emph{$\Tth$-graph morphism} (resp.\ \emph{$\Tth$-functor}) if $Tf\cdot a\le b\cdot f$. Given two $\Tth$-categories $(X,a)$ and $(Y,b)$, a $\Tth$-matrix $\varphi\:X\krelto Y$ is a \emph{$\Tth$-distributor}, denoted as $\varphi\:(X,a)\kmodto (Y,b)$, if $\varphi\kleisli a\le\varphi$ and $b\kleisli \varphi\le \varphi$.
\end{definition}

\begin{proposition}
Let $\Tth$ be a strict topological theory. $\Tth$-graphs and $\Tth$-graph morphisms, resp.\ $\Tth$-categories and $\Tth$-functors, form a category $\Gph{\Tth}$, resp.\ $\Cat{\Tth}$, for the obvious composition and identities. $\Tth$-categories and $\Tth$-distributors between them form a locally ordered category $\Mod{\Tth}$, with the Kleisli convolution as composition and the identity on $(X,a)$ given by $a\:(X,a)\kmodto (X,a)$. 
\end{proposition}

\begin{examples}\label{ExCats} We come back to the theories of Example \ref{ExTheories}:
\begin{enumerate}
\item Trivial theory: For each quantale $\V$, $\Ith_\V$-categories are precisely $\V$-categories and $\Ith_\V$-functors are $\V$-functors. As usual, we write $\Cat{\V}$ instead of $\Cat{\Ith_\V}$, $\Gph{\V}$ instead of $\Gph{\Ith_\V}$, and so on. In particular, $\Cat{\V}$ is the category $\ORD$ of ordered sets if $\V=\two$, and for $\V=[0,\infty]$ one obtains Lawvere's category $\MET$ of generalised metric spaces \cite{Law_Metric}.
\item Ultrafilter theories: The main result of \cite{Bar_RelAlg} states that $\Cat{\Uth_\two}$ is isomorphic to the category $\TOP$ of topological spaces. In \cite{CH_TopFeat} it is shown that $\Cat{\Uth_{[0,\infty]}}$ is isomorphic to the category $\AP$ of approach spaces \cite{Low_ApBook}.
\end{enumerate}
\end{examples}

Since we always have $\varphi\kleisli a\ge\varphi$ and $b\kleisli\varphi\ge\varphi$, the $\Tth$-distributor condition above implies equality. The local order in $\Mod{\Tth}$ is inherited from $\Mat{\V}$, but whereas the latter is a quantaloid (i.e.\ has local suprema which are stable under composition), the former generally is not. In fact, the matrix-infimum of distributors is a distributor, but the matrix-supremum of distributors is not necessarily a distributor. It is easy to see that all liftings (i.e.\ right adjoints to $\psi\kleisli-$) exist in $\Mod{\Tth}$, but the example below shows that extensions (i.e.\ right adjoints to $-\kleisli\psi$) needn't exist.

\begin{lemma}
For any $\psi\:(Y,b)\kmodto (X,a)$ and $(Z,c)$ in $\Mod{\Tth}$\footnote{In fact, this proof also works in the locally ordered category $\UMat{\Tth}$ of so-called \emph{unitary $\Tth$-relations}: its objects are sets and its arrows are those $a\:X\krelto Y$ for which $a\kleisli e_X=a$ holds. Kleisli convolution is composition, and the identity on $X$ is $e_X$.}, the order-preserving map 
\[\psi\kleisli-\:\Mod{\Tth}((Z,c),(Y,b))\to\Mod{\Tth}((Z,c),(X,a))\]
admits a right adjoint.
\end{lemma}
\begin{proof}
For any $\gamma\:(Z,c)\kmodto (X,a)$ we pass from
\begin{align*}
\xymatrix{X & Z\ar@{-_{>}}|-{\object@{o}}[l]_\gamma\\ Y\ar@{-^{>}}|-{\object@{o}}[u]^\psi} &&\text{to}&&
\xymatrix{TX & Z\ar[l]|-{\object@{|}}_\gamma\\
TTX\ar[u]|-{\object@{|}}^{m_X}\\ TY\ar[u]|-{\object@{|}}^{\Txi\psi}}
\end{align*}
and put $\psi\homkleislileft\gamma:=(m_X\cdot\Txi\psi)\homcompleft\gamma$: it is easily verified that $\psi\homkleislileft\gamma$ is a $\Tth$-distributor and satisfies the required universal property.
\end{proof}
\begin{example}
Consider the real numbers with their Euclidian topology, $\R_E$, and with the discrete topology, $\R_D$. Then certainly $f\:\R_D\to\R_E,\;x\mapsto x$ is continuous. Further one checks that a distributor $\theta\:R_E\kmodto E$, resp.\ $\kappa\:R_D\kmodto E$, is ``the same as'' a closed subset of $\R$ 
for the respective topologies, where $E$ denotes a one-element space. Finally, one finds that $\theta\kleisli f_*=\theta$ for any $\theta\:R_E\kmodto E$. Because the supremum of closed subsets in $\R_E$ is in general different 
from their supremum in $\R_D$ (i.e.\ their union), we now find that $-\kleisli f_*$ does not necessarily preserve such suprema.
\end{example}

We shall now establish the expected relation between $\Tth$-functors and $\Tth$-distributors: each $\Tth$-functor induces an adjoint pair of $\Tth$-distributors (see \cite{CH_Compl}).
 
Let $X=(X,a)$ and $Y=(Y,b)$ be $\Tth$-categories and $f\:X\to Y$ be a $\Tth$-functor. We define $\Tth$-distributors $f_*\:X\kmodto Y$ and $f^*\:Y\kmodto X$ by putting $f_*=b\cdot f$ and $f^*=Tf^\circ\cdot b$ respectively. Hence, for $\frx\in TX$, $\fry\in TY$, $x\in X$ and $y\in Y$, $f_*(\fry,x)=b(\fry,f(x))$ and $f^*(\frx,y)=b(Tf(\frx),y)$. One easily verifies the rules
\begin{align*}
f^*\kleisli\varphi&=Tf^\circ\varphi &&\text{and}&& \psi\kleisli f_*=\psi\cdot f,
\end{align*}
for $\Tth$-distributors $\varphi$ and $\psi$, to conclude that $f_*\dashv f^*$ in $\Mod{\Tth}$. One calls a $\Tth$-category $X$ \emph{Cauchy-complete} (Lawvere complete in \cite{CH_Compl}) if every adjunction $\varphi\dashv\psi$ of $\Tth$-distributors $\varphi\:Y\kmodto X$ and $\psi\:X\kmodto Y$ is of the form $f_*\dashv f^*$, for some $\Tth$-functor $f\:Y\to X$. As shown in \cite{CH_Compl}, in order to check if $X$ is Cauchy-complete it is enough to consider the case $Y=(1,k_!)$ where $k_!:=!^\circ\cdot k\:1\relto T1$.

Furthermore, we have functors
\[\xymatrix@=8ex{
\Cat{\Tth}\ar[r]^{(-)_*} & \Mod{\Tth} & \Cat{\Tth}^\op\ar[l]_{(-)^*}},
\]
where $X_*=X=X^*$ for each $\Tth$-category $X=(X,a)$. Hence, $\Cat{\Tth}$ becomes a 2-category via the functor $(-)_*$: we define $f\le g$ if $f_*\le g_*$, which is equivalent to $g^*\le f^*$. Taking this 2-categorical structure into account, the second functor above can be written as $(-)^*\:\Cat{\Tth}^{\co\op}\to\Mod{\Tth}$.

Let us point out some other 2-functors that are of interest (cf.\ the diagram in figure \ref{diag1}):
\begin{itemize}
\item The forgetful $U\:\Cat{\Tth}\hookrightarrow\Gph{\Tth}$ has a left adjoint $F$ which is for instance described in \cite{Hof_Quot}.
\item Each $\Tth$-category $(X,a)$ has an underlying $\V$-category $\ForgetToV(X,a)=(X,e_X^\circ\cdot a)$. This defines a functor $\ForgetToV\:\Cat{\Tth}\to\Cat{\V}$ which has a left adjoint $\ForgetToVAd\:\Cat{\V}\to\Cat{\Tth}$ defined by $\ForgetToVAd(X,r)=(X,\Txi r\cdot e_X)$. 
\item As observed in \cite{CH_Compl}, there is another functor from $\Cat{\Tth}$ to $\Cat{\V}$, namely $\MFunctor\:\Cat{\Tth}\to\Cat{\V}$ which sends a $\Tth$-category $(X,a)$ to the $\V$-category $(TX,m_X\cdot \Txi a)$. This functor shall only be needed to define the dual of a $\Tth$-category (see further) and is not pictured in the diagram.
\item Each Eilenberg--Moore $\mT$-algebra $(X,\alpha)$ can be considered as a $\Tth$-category by regarding the function $\alpha\:TX\to X$ as a $\V$-matrix $\alpha^\circ\:X\relto TX$. This defines a functor $\Discrete\:\SET^\mT\to\Cat{\Tth}$, whose composition with $\SET\to\SET^\mT$ we denote as $|-|\:\SET\to\Cat{\Tth}$.
\end{itemize}
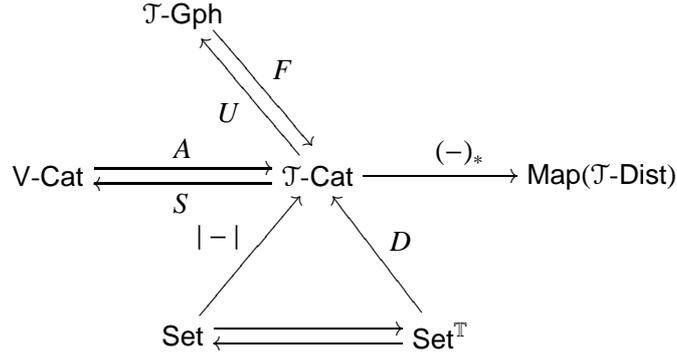
\begin{figure}
$$\xymatrix@R=15mm@C=5mm{
 & \Gph{\Tth}\ar@<2mm>[dr]^F \\
\Cat{\V}\ar@<1mm>[rr]^A & & \Cat{\Tth}\ar[rr]^{(-)_*}\ar[ul]^U\ar@<1mm>[ll]^S & & \Map(\Mod{\Tth}) \\
 & \SET\ar[ur]^{|-|}\ar@<1mm>[rr] & & \SET^\mT\ar@<1mm>[ll]\ar[ul]_D}$$
\caption{Some (2-)functors of interest}\label{diag1}
\end{figure}
We shall now discuss some further properties of these functors, especially concerning monoidal structure.

The tensor product $\otimes$ on $\V$ has a canonical lifting to $\Mat{\V}$:one puts $X\otimes Y=X\times Y$, and for $\V$-matrices $a\:X\to Y$ and $b\:X'\to Y'$ one defines
\[
(a\otimes b)((x,x'),(y,y'))=a(x,y)\otimes b(y,y').
\]
Clearly, any one element set $1$ is neutral for this tensor product. Then $\Txi\:\Mat{\V}\to\Mat{\V}$ together with the natural transformation $m\:\Txi\Txi\natto\Txi$ and the op-lax natural transformation $\Id\natto\Txi$ becomes a \emph{lax Hopf monad} on $(\Mat{\V},\otimes,1)$ in the sense that we have maps $\tau_{X,Y}\:T(X\times Y)\to TX\times TY$ and $!\:T1\to 1$ so that the diagrams
\begin{align*}
\xymatrix{T(X\otimes Y)\ar[d]|-{\object@{|}}_{\Txi(r\otimes s)}\ar[r]^{\tau_{X,Y}} & TX\otimes TY\ar[d]|-{\object@{|}}^{\Txi r\otimes \Txi s}\\
T(X'\otimes Y')\ar[r]_{\tau_{X',Y'}} & TX'\otimes TY'}
&&\text{and}&&
\xymatrix{T1\ar[r]^{!}\ar[d]|-{\object@{|}}_{\Txi k} & 1\ar[d]|-{\object@{|}}^k\\ T1\ar[r]_{!} & 1}
\end{align*}
commute in $\Mat{\V}$. Hence, we cannot speak of a Hopf monad (see \cite{Moe_Tensor}) only because $(\Txi,m,e)$ is just a lax monad on $\Mat{\V}$. This additional structure permits us to turn also $\Cat{\Tth}$ into a tensored category: for $\Tth$-categories (or, more general, $\Tth$-graphs) $X=(X,a)$ and $Y=(Y,b)$ we define $X\otimes Y=(X\times Y,c)$ where $c=\tau_{X,Y}^\circ\cdot(a\otimes b)$. Explicitly, for $\frw\in T(X\times Y)$, $x\in X$, $y\in Y$, $\frx=T\pi_1(\frw)$ and $\fry=T\pi_2(\frw)$ we have
\[
c(\frw,(x,y))=a(\frx,x)\otimes b(\fry,y).
\]
The $\Tth$-category $E=(1,k_!)$, with $k_!\:=!^\circ\cdot k\:1\relto T1$, is neutral for $\otimes$. This tensor product and its properties are studied in \cite{Hof_TopTh} (unfortunately, without mentioning the concept of a Hopf monad). The functors introduced above have now the following properties: $\ForgetToVAd\:\Cat{\V}\to\Cat{\Tth}$ and $\MFunctor\:\Cat{\V}\to\Cat{\Tth}$ are op-monoidal 2-functors, $\ForgetToV\:\Cat{\Tth}\to\Cat{\V}$ is a strong monoidal 2-functor and $\Discrete\:\SET^\mT\to\Cat{\Tth}$ is a strong monoidal functor.

Another important feature of a topological theory is that it allows us to consider $\V$ as a $\Tth$-category $\V=(\V,\hom_\xi)$, where
\[
\hom_\xi\:T\V\times\V\to\V,\;\;(\frv,v) \mapsto \hom(\xi(\frv),v).
\]
Note that $\V=(\V,\hom_\xi)$ is in general not isomorphic to $\ForgetToVAd(\V,\hom)$. Furthermore, $\V$ is a monoid in $(\Cat{\Tth},\otimes,E)$ since both $k\:E\to\V$ and $\otimes\:\V\otimes\V\to\V$ are $\Tth$-functors. Through the strong monoidal 2-functor $\ForgetToV\:\Cat{\Tth}\to\Cat{\V}$, this specialises to the usual monoid structure on $\V$ in $\Cat{\V}$. We also remark that $\xi\:|\V|\to\V$ becomes now a $\Tth$-functor.

We finish this section by presenting a characterisation of $\Tth$-distributors as $\V$-valued $\Tth$-functors, generalising therefore a well-known fact about $\V$-categories. This result involves the dual category, a concept which has no obvious $\Tth$-counterpart. However, the following definition (see \cite{CH_Compl}) proved to be useful: given a $\Tth$-category $X=(X,a)$, one puts
\[X^\op=\ForgetToVAd(\MFunctor(X)^\op).\]
\begin{theorem}[\cite{CH_Compl}]\label{CharTMod}
For $\Tth$-categories $(X,a)$ and $(Y,b)$, and a $\Tth$-matrix $\psi\:X\krelto Y$, the following assertions are equivalent.
\begin{eqcond}
\item $\psi\:(X,a)\kmodto(Y,b)$ is a $\Tth$-distributor.
\item Both $\psi\:|Y|\otimes X\to\V$ and $\psi\:Y^\op\otimes X\to\V$ are $\Tth$-functors.
\end{eqcond}
\end{theorem}
In particular, since $a\:X\kmodto X$ for each $\Tth$-category $X=(X,a)$, we have two $\Tth$-functors
\begin{align*}
a\:|X|\otimes X\to\V &&\text{and}&& a\:X^\op\otimes X\to\V.
\end{align*}
The theorem above, together with the condition $T1=1$, can now be used to construct the Cauchy-completion $\yoneda\:X\to\tilde{X}$ of a $\Tth$-category $X$, where $\tilde{X}$ has as objects
\begin{align*}
\{\psi\in\Mod{\Tth}(E,X)\mid\psi\text{ is left adjoint}\}
\end{align*}
and $\yoneda$ is the Yoneda embedding $x\mapsto x_*$. For details we refer to \cite{HT_LCls}.
\begin{examples}
We consider first $\V=[0,\infty]$, hence $\V$-category means (generalised) metric space. In \cite{Law_Metric}  F.W.~Lawvere has shown that equivalence classes of Cauchy sequences in a metric space $X$ correspond precisely to left adjoint $[0,\infty]$-distributors $\psi\:E\modto X$, and a Cauchy sequence converges to $x$ if and only if $x$ is a colimit of the corresponding $[0,\infty]$-distributor. Hence, Cauchy completeness has the usual meaning and $\tilde{X}$ describes the usual Cauchy completion of a metric space.

In \cite{HT_LCls} it is shown that the topological space (= $\Uth_\two$-category) $\tilde{X}$ is homeomorphic to the space of all completely prime filters on the lattice $\tau$ of open subsets of $X$, and $\yoneda\:X\to\tilde{X}$ corresponds to the map which sends $x\in X$ to its neighbourhood filter. Of course, one can equivalently consider right adjoint $\Uth_\two$-distributors $\varphi\:X\kmodto E$, and in \cite{CH_Compl} it is shown that a $\Uth_\two$-distributors $\varphi\:X\kmodto E$ is right adjoint if and only if $\varphi\:X\to\two$ is the characteristic map of an irreducible closed subset $A$ of $X$, and $\varphi=x^*$ if and only $A=\overline{\{x\}}$. Hence, a topological space $X$ is Cauchy complete if and only if $X$ is weakly sober. 
\end{examples}

\section{$\Tth$-categories versus $\Tth$-frames}\label{Sec2}

Recall from the Introduction that $\Omega\:\TOP^\op\to\FRM$ is the functor that sends any topological space $X$ to the frame $\Omega(X)$ of its open subsets, and any continuous function $f\:X\to Y$ to the frame homomorphism $\Omega(f)\:\Omega(Y)\to\Omega(X)$ given by inverse image. It is straightforward to see that $\Omega(X)$ is isomorphic (qua ordered set) to $\TOP(X,S)$, where $S$ is the Sierpinski space, topological spaces are considered with their specialisation order (which continuous functions preserve), and $\TOP(X,S)$ is ordered pointwise. In fact, modulo these isomorphisms, $\Omega\:\TOP^\op\to\FRM$ is simply a corestriction of the representable functor $\TOP(-,S)\:\TOP^\op\to\ORD$. Further recall that the left adjoint to $\Omega$, $\pt\:\FRM\to\TOP^\op$, is also defined by means of a representable, namely $\FRM(-,\two)$. It is now noteworthy that the specialisation order of the Sierpinski space $S$ is precisely the two-element chain $\two=\{0\leq 1\}$, and conversely $S$ is the Alexandrov topology on $\two$. For this reason, some have called the two-point set $\{0,1\}$ a {\em dualising object} in this situation: it can be endowed with two different structures, the Sierpinski topology and the total order, making it objects of two different categories, topological spaces and frames, and represents a duality between these categories \cite{PT_Dual}.

This analysis now suggests our method to define the category of ``$\Tth$-frames'' in the general context of $\Tth$-categories, as follows. For any strict topological theory $\Tth=(\mT,\V,\xi)$, the quantale $\V$ naturally bears the structure of a $\Tth$-category; thus we have the representable functor $\Cat{\Tth}(-,\V)\:\Cat{\Tth}^\op\to\Cat{\V}$. Now we devise a category $\Frm{\Tth}$ of ``$\Tth$-frames'' and ``$\Tth$-frame homomorphisms'' in such a way\footnote{At this point we should mention that, specialised to the topological case, $\V=\two$ becomes the Sierpinski space with $\{1\}$ closed so that $\TOP(X,\two)$ is naturally isomorphic to the \emph{co-frame} of closed subsets of $X$. Nevertheless, we prefer to use the term ``$\Tth$-frame'' in the sequel.} that (i) the representable $\Cat{\Tth}(-,\V)\:\Cat{\Tth}^\op\to\Cat{\V}$ corestricts to a functor $\Omega\:\Cat{\Tth}^\op\to\Frm{\Tth}$, and (ii) $\V$ is an object of $\Frm{\Tth}$ representing a functor $\pt\:\Frm{\Tth}\to\Cat{\Tth}^\op$. Ideally, these functors should then be adjoint, but for now we are unable to prove this. However, we do show a natural comparison $\eta_X\:X\to\pt(\Omega(X))$ for any $\Tth$-category $X$, and we prove that $X$ is a {\em Cauchy complete} $\Tth$-category (amounting to {\em sobriety} in the case of $T_0$ topological spaces) if and only if $\eta_X$ is surjective.

For technical reasons {\bf we shall from now on assume that $T1=1$}. Together with Theorem \ref{CharTMod} this implies that a $\Tth$-distributor $\varphi\:X\kmodto E$ is ``the same thing as'' a $\Tth$-functor $\varphi\:X\to\V$ (recall that $E$ is the unit for the tensor product in $\Cat{\Tth}$). Furthermore, for any $\alpha\:X\krelto E$,
\[
-\kleisli\alpha\:\Mat{\Tth}(E,E)\to\Mat{\Tth}(X,E)
\]
has a right adjoint $(-)\homkleisliright\alpha$ calculated as in $\Mat{\V}$, due to $\Txi v=v$ (for any $v\:1\relto 1$). Unfortunately, the condition $T1=1$ excludes Example \ref{ExTheories} \eqref{WordTh}. 
\begin{lemma}
The following assertions hold.
\begin{enumerate}
\item $\bigwedge\:\V^I\to\V$ is a $\Tth$-functor, for each index set $I$.
\item $\hom(v,-)\:\V\to\V$ is a $\Tth$-functor, for each $v\in\V$.
\item $v\otimes-\:\V\to\V$ is a $\Tth$-functor, for each $v\in\V$.
\end{enumerate}
\end{lemma}
It is now straightforward that the representable functor $\Cat{\Tth}(-,\V)\:\Cat{\Tth}^\op\to\ORD$ lifts to a functor $\V^-\:\Cat{\Tth}^\op\to\Cont{\V}$ by putting $\V^X$ to be the full sub-$\V$-category of $P(SX)$ (i.e.\ the usual $\V$-category of covariant $\V$-presheaves on $SX$, the ``specialisation'' $\V$-category underlying the $\Tth$-category $X$) determined by the elements in $\Cat{\Tth}(X,\V)$. Clearly, a $\Tth$-functor $f\:X\to Y$ (with $Y$ being a $\Tth$-category) induces a $\V$-functor $\V^f\:\V^Y\to\V^X$ which preserves infima, tensors and cotensors, i.e.\ all weighted limits. Being complete, $\V^Y$ is also cocomplete, but suprema are typically not computed pointwise and hence in general not preserved by $\V^f$. However, a particular class of suprema are preserved by $\V^f$, as we show next. 
\begin{proposition}[\cite{Hof_TopTh}]\label{Compact}
Let $X$ be a $\Tth$-category. Then $X$ is compact if and only if $\bigvee\:\V^X\to\V$ is a $\Tth$-graph morphism. In particular, $\bigvee\:\V^X\to\V$ is a $\Tth$-functor for each $\mT$-algebra $X$.
\end{proposition}
Here a $\Tth$-category $X=(X,a)$ is called compact if $k\le\bigvee\{a(\frx,x)\mid x\in X\}$, for every $\frx\in UX$. For topological spaces, compact has the usual meaning, and an approach space is compact if and only if its measure of compactness is $0$ (see \cite{Low_ApBook}). Also note that every $\V$-category is compact. In the proposition above, $\V^X$ is the $\Tth$-graph with structure matrix $\fspstr{-}{-}$ defined as
\[
\fspstr{\frp}{\varphi}=\hspace{-2ex}\bigwedge_{\substack{\frq\in T(X\times Y^X),x\in X\\ \frq\mapsto \frp}}\hspace{-2ex}\hom(a(T\pi_1(\frq),x),\hom(\xi\cdot T\!\ev(\frq),\varphi(x))).
\]
In fact, we apply here to $\V$ the right adjoint $(-)^X$ of $X\otimes-\:\Gph{\Tth}\to\Gph{\Tth}$ (see \cite{Hof_TopTh}). Note that we use here the same notation $\V^X$ for the $\Tth$-graph and the $\V$-category (defined on the same set of objects). However, if $\frp=e_{\V^X}(\varphi')$ with $\varphi'\in\V^X$ in the formula above, then
\[
\fspstr{e_{V^X}(\varphi')}{\varphi}=\bigwedge_{x\in X}\hom(\varphi'(x),\varphi(x))=\varphi\homkleisliright\varphi'=[\varphi',\varphi],
\]
i.e.\ the underlying $\V$-graph of the $\Tth$-graph $\V^X$ is actually the $\V$-category $\V^X$ described above. As a consequence, if $\varphi'\:X\kmodto E$ has a left adjoint $\psi\:E\kmodto X$ in $\Mod{\Tth}$, then
\[\fspstr{e_{Y^X}(\varphi')}{\varphi}=\varphi\kleisli\psi.\]

Proposition \ref{Compact} suggests now the following new notions. 

\begin{definition}\label{T-suprema}
 Let $A=(A,a)$ and $B=(B,b)$ be $\Tth$-graphs whose underlying $\V$-graphs are $\V$-categories; for shorthand, we will write $A$ (resp.\ $B$) for both the $\Tth$-graph and the underlying $\V$-category. A $\V$-functor $f\:A\to B$ is said to be \emph{$\Tth$-compatible} if, for each $\mT$-algebra $I$ and each $\Tth$-graph morphism $h\:I\to A$, the composite $f\cdot h$ is a $\Tth$-graph morphism as well. By a \emph{$\Tth$-diagram} in $A$ we mean a $\Tth$-graph morphism $D\:I\to A$ where $I$ is a $\mT$-algebra; and a supremum of a $\Tth$-diagram is a \emph{$\Tth$-supremum}. Finally, we say that a $\Tth$-compatible $\V$-functor $\Phi\:A\to B$ \emph{preserves $\Tth$-suprema} if $\Phi$ preserves suprema of $\Tth$-diagrams. 
\end{definition}
\begin{proposition}
For every $\Tth$-functor $f\:X\to Y$, the $\V$-functor $\V^f\:\V^Y\to\V^X$ underlies a $\Tth$-graph morphism, and hence is  $\Tth$-compatible. Moreover, $\V^f$ preserves $\Tth$-suprema (but in general not all suprema).
\end{proposition}
\begin{remark}
We consider the Yoneda morphism $\yoneda\:X\to\tilde{X}$. Then $\V^{\yoneda}\:\V^{\tilde{X}}\to\V^X$ is an isomorphism of $\V$-categories, where $\V^{\yoneda}$ sends $\tilde{\varphi}\:\tilde{X}\to\V$ to its restriction $\varphi\:X\to\V$. In fact, when considering $\varphi$, $\tilde{\varphi}$ as $\Tth$-distributors $\varphi\:E\kmodto X$ and $\tilde{\varphi}\:E\kmodto\tilde{X}$, we have $\tilde{\varphi}=\yoneda_*\kleisli\varphi$ resp.\ $\varphi=\yoneda^*\kleisli\tilde{\varphi}$. Hence, since $\yoneda_*\dashv\yoneda^*$ is an equivalence of $\Tth$-distributors,
\[
\tilde{\varphi}\homkleisliright\tilde{\psi}=\varphi\homkleisliright\psi
\]
for all $\tilde{\varphi},\tilde{\psi}\:\tilde{X}\to\V$. Its inverse $\Phi\:\V^X\to\V^{\tilde{X}},\,\varphi\to\tilde{\varphi}$ certainly preserves all suprema. Moreover, $\Phi$ is $\Tth$-compatible. To see this, let $h\:I\to\V^X$ be a $\Tth$-graph morphism where $I$ is a $\Tth$-algebra. Since $\V$ is injective with respect to fully faithful $\Tth$-functors, we have an (in fact unique) extension $l\:\tilde{X}\otimes I\to\V$ of $\umate{h}\:X\otimes I\to\V$ along $\yoneda\otimes\id_I\:X\otimes I\to\tilde{X}\otimes I$. Then $\mate{l}(i)\cdot\yoneda=h(i)$ for each $i\in I$, and therefore $\mate{l}=\Phi\cdot h$.
\end{remark}
\begin{definition}\label{T-colimit}
Assume that the underlying $\V$-category of $A$ has all tensors. A \emph{$\Tth$-weighted} diagram in $A$ is given by a set $I$ together with a $\mT$-algebra structure $\alpha\:TI\to I$ and a $\V$-category structure $r\:I\relto I$, a $\V$-functor $h\:I\to A$ and a $\V$-distributor $\psi\:1\modto I$ such that the map
\[
I\to A,\;i\mapsto \psi(i)\otimes h(i)
\]
is a $\Tth$-graph morphism. The colimit of a $\Tth$-weighted diagram in $A$ is called a \emph{$\Tth$-colimit}, and $A$ is \emph{$\Tth$-cocomplete} if all $\Tth$-colimits exist in $A$. A $\V$-distributor $\varphi\:1\modto A$ is called \emph{$\Tth$-generated} if $\varphi=h_*\cdot\psi$ in $\Mod{\V}$, for $h$ and $\psi$ as above.
\end{definition}
\begin{proposition}\label{T-cocomplete}
The following assertions are equivalent, for a $\Tth$-graph $A=(A,a)$ where $SA$ is a $\V$-category. 
\begin{eqcond}
\item $A$ is $\Tth$-cocomplete.
\item $A$ has all tensors and all $\Tth$-suprema.
\item Each $\Tth$-generated $\V$-distributor $\varphi\:1\modto A\in P(A)$ has a supremum $\ssup_A(\varphi)$ in $A$.
\end{eqcond}
\end{proposition}
Note that a $\Tth$-compatible and tensor-preserving $\V$-functor $f\:A\to B$ sends $\Tth$-weighted diagrams in $A$ to $\Tth$-weighted diagrams in $B$. In fact, we have
\begin{proposition}
Let $A$ and $B$ be $\Tth$-graphs whose underlying $\V$-categories are $\Tth$-cocomplete, and let $f\:A\to B$ be a $\Tth$-compatible $\V$-functor which preserves tensors. Then $f$ preserves $\Tth$-colimits if and only if $f$ preserves $\Tth$-suprema.
\end{proposition}
Based on the considerations above, we now propose a $\Tth$-equivalent for the concept of a co-frame (but note that we call these ``$\Tth$-frames'' and not ``$\Tth$-co-frames''):
\begin{definition}\label{TFrm}
$\Frm{\Tth}$ is the locally ordered category with:
\begin{description}
\item[objects] {\em $\Tth$-frames}, i.e.\ $\Tth$-graphs $A$ whose underlying $\V$-graph is a complete $\V$-category satisfying the following {\em distributivity law}: for any distributor $\varphi\:I\modto 1$ and functor $h\:I\to PA$ such that $h(i)$ is $\Tth$-generated for all $i\in I$, if $\lim(\varphi,h)$ is $\Tth$-generated then $\ssup_A(\lim(\varphi,h))=\lim(\varphi,\ssup_A\!\cdot\ h)$.
\item[morphisms] {\em $\Tth$-frame homomorphisms}, i.e.\ $\Tth$-compatible $\V$-functors between the underlying $\V$-categories of $\Tth$-frames, that furthermore preserve weighted limits and $\Tth$-weighted colimits.
\end{description}
\end{definition}
By construction, we have a canonical forgetful functor $\Frm{\Tth}\to\Cont{\V}$. Earlier we already explained that $\V\in\Cat{\Tth}$ and that the representable functor $\Cat{\Tth}(-,\V)\:\Cat{\Tth}^\op\to\ORD$ lifts to a functor $\V^-\:\Cat{\Tth}^\op\to\Cont{\V}$. Now we can prove:
\begin{corollary}\label{functor1}
The functor $\V^-\:\Cat{\Tth}^\op\to\Cont{\V}$ factors through the forgetful functor $\Frm{\Tth}\to\Cont{\V}$; we call the resulting functor $\Omega\:\Cat{\Tth}^\op\to\Frm{\Tth}$.
\end{corollary}
\begin{proof}
For each $\Tth$-functor $f\:X\to Y$, the underlying $\V$-graph of the $\Tth$-graph $\V^X$ is a complete $\V$-category, and $\V^f\:\V^Y\to\V^X$ is a $\Tth$-graph morphism which preserves all weighted limits and all $\Tth$-weighted colimits. Furthermore $A=\V^X$ satisfies the distributivity axiom in Definition \ref{TFrm} since the presheaf $\V$-category $P(SX)$ is completely distributive, and $A$ is closed in $P(SX)$ under weighted limits and $\Tth$-weighted colimits.
\end{proof}
Since $\V\in\Frm{\Tth}$, we certainly have a representable functor $\Frm{\Tth}(-,\V)\:(\Frm{\Tth})^\op\to\ORD$. But there is more:
\begin{corollary}\label{functor2}
The functor $\Frm{\Tth}(-,\V)\:(\Frm{\Tth})^\op\to\ORD$ lifts to a functor $\pt\:(\Frm{\Tth})^\op\to\Cat{\Tth}$.
\end{corollary}
\begin{proof}
This is done by putting on $\Frm{\Tth}(X,\V)$ the largest $\Tth$-category structure that makes all evaluation maps $\ev_{X,x}\:\Frm{\Tth}(X,\V)\to\V,\;h\mapsto h(x)$ into $\Tth$-functors.
\end{proof}
Note how, in the two previous corollaries, $\V$ plays the role of a {\em dualising object}: it is on the one hand an object of $\Cat{\Tth}$, and as such represents the functor $\Omega\:\Cat{\Tth}^\op\to\Frm{\Tth}$; but it is also an object of $\Frm{\Tth}$, and as such represents the functor $\pt\:(\Frm{\Tth})^\op\to\Cat{\Tth}$. Next we observe:
\begin{proposition}\label{nattrans}
There is a natural transformation $\eta\:\Id\natto\pt\cdot\Omega$ with components
\[
\eta_X\:X\to\pt(\Omega(X)),\;x\mapsto\ev_{X,x}\mbox{ \ \ for \ \ }X\in\Cat{\Tth}.
\]
\end{proposition}
We do not know whether $\eta_X$ is always fully faithful, but we do have the following result (recall that $E$ is the unit for the tensor in $\Cat{\Tth}$):
\begin{theorem}\label{MainThm}
For any $X\in\Cat{\Tth}$, $\pt(\Omega(X))$ has the same objects as the Cauchy completion $\tilde{X}$ of $X$. In fact, we have an isomorphism $\Map(\Mod{\Tth})(E,X)\to\Frm{\Tth}(\Omega(X),\V)$ of ordered sets, making the diagram
\[
\xymatrix@C=1ex@R=8ex{ & X\ar[dl]_{(-)_*}\ar[dr]^{\eta_X}\\ \Map(\Mod{\Tth})(E,X)\ar[rr] && \Frm{\Tth}(\Omega(X),\V)}
\]
commute. Hence $X$ is Cauchy complete if and only if $\eta_X$ is surjective.
\end{theorem}
The proof of the theorem above is the combination of the results below.
\begin{lemma}
Let $X=(X,a)$ be a $\Tth$-category and $\varphi\:X\to\V$ be a $\Tth$-functor. Then the representable $\V$-functor $\Phi=[\varphi,-]\:\Omega(X)\to\V$ is also a $\Tth$-graph morphism and preserves infima and cotensors. Moreover, if $\psi\dashv\varphi$ in $\Mod{\Tth}$, then $\Phi$ preserves also tensors and $\Tth$-suprema.
\end{lemma}
\begin{proof}
Being a representable $\V$-functor, $\Phi$ preserves infima and cotensors. To see that $\Phi$ is a $\Tth$-graph morphism, recall first that
\[
[\varphi,\varphi']=\bigwedge_{x\in X}\hom(\varphi(x),\varphi'(x)).
\]
Since $\bigwedge\:\V^{X_D}\to\V$ (with $X_D=(X,e_X)$ being the discrete $\Tth$-category) is a $\Tth$-graph morphism, it is enough to show that
\[
\Psi\:\V^X\to\V^{X_D},\;\varphi'\mapsto \hom(\varphi(-),\varphi'(-))
\]
is a $\Tth$-graph morphism. But $\Psi$ is just the mate of the composite
\[
X_D\otimes\V^X\xrightarrow{\Delta\otimes\id}X_D\otimes X\otimes\V^X\xrightarrow{\varphi\otimes\ev}\V_D\otimes\V\xrightarrow{\hom}\V
\]
of $\Tth$-graph morphisms.

Assume now $\psi\dashv\varphi$ in $\Mod{\Tth}$. Then, for any $\varphi'\:X\to\V$ and $v\in\V$, 
\[
[\varphi,v\otimes\varphi']=(v\otimes\varphi')\kleisli\psi
=(v\kleisli\varphi')\kleisli\psi
=v\kleisli(\varphi'\kleisli\psi)
=v\otimes[\varphi,\varphi'].
\]
Finally, to see that $[\varphi,-]$ preserves $\Tth$-suprema, we assume $X$ to be Cauchy complete. Let $D\:I\to\V^X,\,i\mapsto\varphi_i$ be a $\Tth$-diagram. Then, since $\varphi=a(e_X(x),-)$ for some $x\in X$,
\[
[\varphi,\bigvee_{i\in I}\varphi_i]
=[a(e_X(x),-),\bigvee_{i\in I}\varphi_i]
=\left(\bigvee_{i\in I}\varphi_i\right)(x)
=\bigvee_{i\in I}\varphi_i(x)=\bigvee_{i\in I}[\varphi,\varphi_i].\qedhere
\]
\end{proof}
Hence $\psi\mapsto[\varphi,-]$ where $\psi\dashv\varphi$ defines a map $\Map(\Mod{\Tth})(E,X)\to\Frm{\Tth}(\Omega(X),\V)$, which is clearly injective and hence, by definition, an order-embedding. Before stating our next result, we recall that $\varphi=\bigvee_{\frx\in TX}(a(\frx,-)\otimes\xi\cdot T\varphi(\frx))$ for each $\Tth$-functor  $\varphi\:X\to\V$.
\begin{proposition}\label{FrmMorph}
Let $X=(X,a)$ be a $\Tth$-category and $\Phi\:\Omega(X)\to\V$ be a $\V$-functor. Then the following assertions are equivalent.
\begin{eqcond}
\item $\Phi=[\varphi,-]$ for some right adjoint $\Tth$-distributor $\varphi\:X\kmodto E$.
\item $\Phi$ preserves infima, tensors, cotensors and $\Tth$-suprema.
\item $\Phi$ preserves infima, tensors, cotensors and, for each $\varphi\in\V^X$,
\begin{align}\label{CompCond}
\Phi(\varphi)=\bigvee_{\frx\in TX}\Phi(a(\frx,-)\otimes\xi\cdot T\varphi(\frx)).\tag{*}
\end{align}
\end{eqcond}
\end{proposition}
\begin{proof}
(i)$\Rightarrow$(ii): Follows from the lemma above.\\
(ii)$\Rightarrow$(iii): It is enough to observe that
\[
|X|\otimes X\to\V,\;(\frx,x)\mapsto a(\frx,x)\otimes\xi\cdot T\varphi(\frx)
\]
is a $\Tth$-functor since it can be written as the composite
\[
|X|\otimes X\xrightarrow{\Delta\otimes\id_X}|X|\otimes|X|\otimes X \xrightarrow{T\varphi\otimes a}|\V|\otimes\V\xrightarrow{\xi\otimes\id_X}\V\otimes\V\xrightarrow{\otimes}\V
\]
of $\Tth$-functors.\\
(iii)$\Rightarrow$(i): Since $\Phi\:\V^X\to\V$ preserves infima and cotensors, $\Phi$ is representable by some $\varphi\in\V^X$, i.e.\ $\Phi=[\varphi,-]$. Hence, by the lemma above, $\Phi$ is a $\Tth$-graph morphism. We put $\psi\:=\Phi\cdot\mate{a}$,
\[
\xymatrix{|X|,X^\op\ar[r]^-{\mate{a}}\ar[rd]_\psi &\V^X\ar[d]^\Phi\\ & \V}
\]
then $\psi\:E\kmodto X$ is a $\Tth$-distributor by Theorem \ref{CharTMod}. We have, for any $\frx\in TX$ and $x\in X$,
\[
\psi(\frx)\otimes\varphi(x)
=[\varphi,a(\frx,-)]\otimes\varphi(x)
\le\hom(\varphi(x),a(\frx,x))\otimes\varphi(x)\le a(\frx,x).
\]
On the other hand, 
\begin{align*}
\bigvee_{\frx\in TX}\psi(\frx)\otimes\xi\cdot T\varphi(\frx)
&=\bigvee_{\frx\in TX}[\varphi,a(\frx,-)]\otimes\xi\cdot T\varphi(\frx)\\
&=\bigvee_{\frx\in TX}[\varphi,a(\frx,-)\otimes\xi\cdot T\varphi(\frx)]\\
&=[\varphi,\bigvee_{\frx\in TX}a(\frx,-)\otimes\xi\cdot T\varphi(\frx)] \\
&=[\varphi,\varphi] \\
&\ge k,
\end{align*}
we have shown that $\psi\dashv\varphi$.
\end{proof}
We conclude that the map $\Map(\Mod{\Tth})(E,X)\to\Frm{\Tth}(\Omega(X),\V),\;\psi\mapsto[\varphi,-]$ is actually bijective. Finally, for any $x\in X$ and each $\Tth$-functor $\varphi\:X\to\V$, we have
\[
[x^*,\varphi]=\varphi\kleisli x_*=\bigvee_{\frx\in TX}a(\frx,x)\otimes\xi\cdot T\varphi(\frx)=\varphi(x)=\ev_{X,x}(\varphi),
\]
which proves the commutativity of the diagram in Theorem \ref{MainThm}.

\section{Examples}

We consider first the identity theory for an arbitrary quantale $\V$, cf.\ Example \ref{ExTheories} (1); in this case, $\Cat{\Tth}=\Cat{\V}$ is the category of $\V$-enriched categories. A $\Tth$-diagram is just an ordinary diagram, and therefore $\Frm{\Tth}$ is the 2-category having as objects complete (and cocomplete) completely distributive $\V$-categories, and as morphisms all limit- and colimit-preserving functors between them. Writing $\Frm{\V}$ for this category, we {\em do} have an adjunction
\[
(\Cat{\V})^\op\adjunction{\pt}{\Omega}\Frm{\V},
\]
and $\eta_X\:X\to\pt(\Omega(X))$ is fully faithful for each $\V$-category $X$. The latter is a consequence of the well-known fact that $\V$ is \emph{initially dense} in $\Cat{\V}$ (see \cite{T_Haute_Bodeux}, for instance), i.e.\ for each $\V$-category $X$ the source $\Cat{\V}(X,\V)$ is initial (jointly fully faithful).

In particular, for $\V=\two$ (the two-element chain), the adjunction above specialises to ordered sets and completely distributive complete lattices
\[
\ORD^\op\adjunction{\pt}{\Omega}\catfont{CCD}
\]
which restricts to a dual equivalence between $\ORD$ and the category $\catfont{TAL}$ of totally algebraic complete lattices and suprema and infima preserving maps. And for $\V=[0,\infty]$ (the extended non-negative real numbers) we obtain an adjunction
\[
\MET^\op\adjunction{\pt}{\Omega}\catfont{CDMet}
\]
where $\catfont{CDMet}$ denotes the category of completely distributive metric spaces and limit- and colimit-preserving contraction maps. This adjunction restricts to a dual equivalence between the full subcategories of {\em Cauchy complete metric spaces} and totally algebraic metric spaces respectively. Here a metric space $X$ is completely distributive if it is cocomplete and the left adjoint $S\:[0,\infty]^{X^\op}\to X$ of the Yoneda embedding $\yoneda_X\:X\to [0,\infty]^{X^\op}$ has a further left adjoint $t_X\:X\to [0,\infty]^{X^\op}$. Furthermore, a completely distributive metric space $X$ is totally algebraic if $S$ restricts to an isomorphism $[0,\infty]^{A^\op}\cong X$ where $A\hookrightarrow X$ is the equaliser of $\yoneda_X$ and $t_X$. We refer to \cite{Stu_Dynamics} where complete distributivity and algebraicity are investigated in the context of quantaloid-enriched categories.

Finally, in the remainder of this section we consider the ultrafilter theories of Example \ref{ExTheories} (2--4); below we denote such a theory as $\Uth$. As recalled in Example \ref{ExCats}, if the underlying quantale is $\V=\two$, then $\Cat{\Uth_{\two}}=\TOP$ is the category of topological spaces; and if $\V=[0,\infty]$, then $\Cat{\Uth_{[0,\infty]}}=\AP$ is the category of approach spaces. Our aim is to show how, in general, the notion of $\Uth$-supremum captures precisely a finiteness condition; so that, consequently, the distributivity law for a $\Uth$-frame expresses that finite suprema must distribute over arbitrary infima. To prove this, we start with a well-known lemma providing a crucial tool when working with ultrafilters (a proof can be found in \cite{Joh_StoneSp}, for instance):
\begin{lemma}\label{Lemma_ExtExcl}
Let $X$ be a set, $\frj$ be an ideal and $\frf$ be a filter on $X$ with $\frf\cap\frj=\varnothing$. Then there exists an ultrafilter $\frx\in UX$ with $\frf\subseteq\frx$ and $\frx\cap\frj=\varnothing$.
\end{lemma}
\noindent
For any $\Uth$-category $X=(X,a)$ and any $A\subseteq X$, the map $\varphi_A\:X\to\V\:x\mapsto\bigvee\{a(\fra,x)\mid\fra\in U(X),A\in\fra\}$
is in fact a $\Uth$-functor: for it is the composite 
$$X\xrightarrow{\mate{a}}\V^{|X|}\to\V^{|A|}\xrightarrow{\bigvee}\V$$ 
of $\Uth$-functors. Note also that, for $\frx\in UX$ and $A\in\frx$, $\xi\cdot U\varphi_A(\frx)\ge k$. Recall further that $\Uxi$ is the lax extension of the ultrafilter monad to $\Mat{\V}$. In the following proofs we shall write $\ll$ for the totally below relation of $\V$. 
\begin{lemma}
Let $X=(X,a)$ be a $\Uth$-category. For $\frx\in UX$ and $x\in X$,
\[\bigwedge\{\varphi_A(x)\mid A\in\frx\}=\bigvee\{\Uxi a(\frX,\doo{x})\mid\frX\in UUX,\,m_X(\frX)=\frx\}.\]
\end{lemma}
\begin{proof}
Clearly, $\Uxi a(\frX,\doo{x})\le a(\frx,x)\le \hom(\xi\cdot U\varphi_A(\frx),\varphi_A(x))\le\varphi_A(x)$ for $\frX\in UUX$ with $m_X(\frX)=\frx$ and $A\in\frx$. Let now $u\in\V$ with $u\ll\bigwedge_{A\in\frx}\varphi_A(x)$. Putting $\frj=\{\calB\subseteq UX\mid \forall\fry\in\calB\,.\,u\not\le a(\fry,x)\}$ defines an ideal disjoint from $\frx^\#=\{UA\mid A\in\frx\}$. Let $\frX\in UUX$ be an ultrafilter with $\frx^\#\subseteq\frX$ and $\frX\cap\frj=\varnothing$. Then $m_X(\frX)=\frx$ and
$\Uxi a(\frX,\doo{x})=\bigwedge_{\calA\in\frX}\bigvee_{\fra\in\calA}a(\fra,x)\ge u$.
\end{proof}
\begin{corollary}
Let $X=(X,a)$ be a $\Uth$-category, $\frx\in UX$ and $x\in X$. Then $a(\frx,x)=\bigwedge\{\varphi_A(x)\mid A\in\frx\}$.
\end{corollary}
\begin{proof}
Because $a(\frx,x)=\bigvee\{\Uxi a(\frX,\doo{x})\mid\frX,\,m_X(\frX)=\frx\}=\bigwedge\{\varphi_A(x)\mid A\in\frx\}$.
\end{proof}
\begin{corollary}
For each $\Uth$-category $X$, the source $\Cat{\Uth}(X,\V)$ is initial (i.e.\  jointly fully faithful).
\end{corollary}
\noindent
We can now show how the ultrafilter monad allows us to capture a finiteness condition, under some strong assumptions on $\V$:
\begin{proposition}\label{ultrafilter_case}
Assume that the quantale $\V$ satisfies $\top=k$, $\{u\in\V\mid u\ll k\}$ is directed and $k\le u\vee v$ implies $k\le u$ or $k\le v$, for all $u,v\in\V$. Let $\Phi\:\Omega(X)\to\V$ be a $\V$-functor which preserves infima, tensors and cotensors. Then $\Phi$ preserves $\Uth$-suprema if and only if $\Phi$ preserves finite suprema.
\end{proposition}
\begin{proof}
Clearly, if $\Phi$ preserves $\Uth$-suprema then $\Phi$ preserves finite suprema. Assume now that $\Phi$ preserves finite suprema, we have to show \eqref{CompCond} in Proposition \ref{FrmMorph}. Note that $\Phi$ is necessarily of the form $\Phi=[\varphi,-]$, for some $\varphi\in\Omega(X)$. Furthermore, by our conditions on $\V$ and since $\Phi$ preserves finite suprema, $\varphi\le\varphi_1\vee\varphi_2$ implies $\varphi\le\varphi_1$ or $\varphi\le\varphi_2$, for $\varphi_1,\varphi_2\in\Omega(X)$. We start by showing that there exists an ultrafilter $\frx\in UX$ with $\varphi=a(\frx,-)$ and $k=\xi\cdot U\varphi(\frx)$. This generalises a well-known property of irreducible closed subsets of a topological space as well as of approach prime elements in the regular function frame of an approach space (see Proposition 5.7 of \cite{BLO_AFrm}).

To this end, note first that $k=\bigvee\{\varphi(x)\mid x\in X\}$. In fact, with $u=\bigvee\{\varphi(x)\mid x\in X\}$, one has $k=\Phi(\varphi)\le\Phi(u)=u$. Let now $u\ll k$ and put $A_u=\{x\in X\mid u\le\varphi(x)\}$. We show that $\varphi\le\varphi_{A_u}$. Consider the set $A=\{x\in X\mid \varphi(x)\le\varphi_{A_u}(x)\}$ and put $v=\bigvee\{\varphi(x)\mid x\in X,x\notin A\}$. One has $A_u\subseteq A$, and therefore $k\neq v$ since otherwise there would exist some $x\in X$ with $u\le\varphi(x)$ and $x\notin A$. Consequently, $\varphi\not\le\varphi\wedge v$. By definition, $\varphi\le\varphi_{A_u}\vee(\varphi\wedge v)$, and we conclude $\varphi\le\varphi_{A_u}$. We have shown that the filter base
\[
 \frf=\{A_u\mid u\ll k\}
\]
is disjoint from the ideal
\[
 \frj=\{B\mid \varphi\not\le\varphi_B\},
\]
and therefore Lemma \ref{Lemma_ExtExcl} provides us with an ultrafilter $\frx\in UX$ with $\frx\cap\frj=\varnothing$. Hence,
\[
 a(\frx,x)=\bigwedge_{A\in\frx}\varphi_A(x)\ge\varphi(x)
\]
Furthermore, for any $u\ll k$,
\[
\varphi_{A_u}(x)=\bigvee_{\fry\in UA_u}a(\fry,x)\le
\bigvee_{\fry\in UA_u}\hom(\xi\cdot U\varphi(\fry),\varphi(x))\le\hom(u,\varphi(x)),
\]
and therefore
\[
a(\frx,x)\le\bigwedge_{u\ll k}\varphi_{A_u}(x)\le\bigwedge_{u\ll k}\hom(u,\varphi(x))
=\hom(\bigvee_{u\ll k}u,\varphi(x))=\varphi(x).
\]
Let now $\varphi'\in\Omega(X)$. Since $\Phi(\varphi')=[\varphi,\varphi']$, one has $\Phi(\varphi')\otimes\varphi(x)\le\varphi'(x)$ for every $x\in X$. Finally
\[
\Phi(a(\frx,-))\otimes\xi\cdot U\varphi'(\frx)
=\xi\cdot U\varphi'(\frx)
=\bigwedge_{A\in\frx}\bigvee_{x\in A}\varphi'(x)
\ge\bigwedge_{A\in\frx}\bigvee_{x\in A}\Phi(\varphi')\otimes\varphi(x)
\ge\Phi(\varphi')\otimes\xi\cdot U\varphi(\frx)=\Phi(\varphi').\qedhere
\]
\end{proof}
As a consequence, in the situation of the proposition above we can modify our definition of $\Frm{\Uth}$ (see Definition \ref{TFrm}) by replacing $\mU$-algebra with \emph{finite (and hence discrete) $\mU$-algebra} everywhere; and we do not need the $\Uth$-graph structure anymore. 

In particular, for $\V=\two$ we thus find that the objects of $\Frm{\Uth_{\two}}$ are complete ordered sets satisfying the co-frame law, and the morphisms are order-preserving maps which preserve all infima and finite suprema. In other words, we arrive at the usual category of co-frames and co-frame homomorphisms. It is well-known (as we recalled in Section \ref{Sec2}) that it is involved in a dual adjunction with $\Cat{\Uth_{\two}}=\TOP$. The situation is similar for $\V=[0,\infty]$. In this case a $\Uth_{[0,\infty]}$-frame is a complete (that is, one admitting all weighted limits and colimits; not to be confused with Cauchy complete) metric space where ``finite colimits commute with arbitrary limits'', and a homomorphism is a contraction map preserving all limits and finite colimits. Our perspective differs here from \cite{BLO_AFrm} where so-called ``approach frames'' were introduced as certain algebras; so far we do not know if both notions are equivalent and therefore leave this as an open problem.

\def\cprime{$'$}


\end{document}